\documentclass[11pt]{amsart}
\usepackage[affil-it]{}
\title{On the poset and asymptotics of Tesler Matrices}
\pdfoutput=1

\author{Jason O'Neill}
\address{ University of California, Los Angeles}
 
\email[J.O'Neill]{jasonmoneill9@gmail.com}

\usepackage{fullpage,enumerate}
\usepackage{amssymb}
\usepackage{amsfonts,amssymb,amsthm,amsmath,mathrsfs,graphicx,wrapfig}
\usepackage{hyperref}
\usepackage{float} 
\usepackage{indentfirst}
\usepackage{stmaryrd}
\usepackage{lmodern}
\usepackage{xfrac}
\usepackage{color}
\usepackage[export]{adjustbox}

\usepackage[toc,page]{appendix}
\usepackage{lipsum}
\usepackage{cite}

\theoremstyle{plain}
\newtheorem{thm}{Theorem}[section]
\newtheorem{lemma}[thm]{Lemma}

\newtheorem{prop}[thm]{Proposition}
\newtheorem{corollary}[thm]{Corollary}
\newtheorem{conj}[thm]{Conjecture}

\theoremstyle{definition} 
\newtheorem{defn}[thm]{Definition} 
\newtheorem{question}[thm]{Question}
\newtheorem{ex}[thm]{Example}

\theoremstyle{remark} 
\newtheorem*{remark}{Remark} 

\definecolor{darkgreen}{rgb}{0,0.7,0}
\definecolor{purplish}{rgb}{0.5,0,0.8}
\definecolor{cobalt}{rgb}{0.0, 0.28, 0.67}
\definecolor{auburn}{rgb}{0.43, 0.21, 0.1}
\definecolor{red}{rgb}{.99,0,0}

\hypersetup{breaklinks,colorlinks,citecolor=red,linkcolor=purplish}

\begin{document}
\pagenumbering{gobble}

\begin{abstract}
Tesler matrices are certain integral matrices counted by the Kostant partition function and have appeared recently in Haglund's study of diagonal harmonics. In 2014, Drew Armstrong defined a poset on such matrices and conjectured that the characteristic polynomial of this poset is a power of $(q-1)$. We use a method of Hallam and Sagan to prove a stronger version of this conjecture for posets of a certain class of generalized Tesler matrices. We also study bounds for the number of Tesler matrices and how they compare to the number of parking functions, the dimension of the space of diagonal harmonics.
\end{abstract}

\maketitle 
\pagenumbering{arabic} 

\section{Introduction} 
Tesler matrices were introduced by Glenn Tesler to study Macdonald polynomials. They have been recently studied due to their relationship with diagonal harmonics and Haglund proved in \protect\cite{JHAG2} that the bigraded Hilbert series for the space of diagonal harmonics, denoted $DH_n$, is the sum over Tesler matrices of a bivariate weight.

\begin{equation} \label{tm_sum}
\text{Hilb}( DH_n;q,t) = \sum_{A} \text{wt}_{q,t} (A)
\end{equation}

\noindent where $A=(a_{i,j})$ is a Tesler matrix and the weight $\text{wt}_{q,t}(\cdot)$ is  
\begin{equation}\label{eq:wt_def}
\text{wt}_{q,t}( A ):= (-M)^{ | \{a_{i,j} >0\} | -n} \prod\limits_{a_{i,j}>0} [a_{i,j}]_{q,t} \text{ with  }   M= \frac{t-1}{q-1} \text{ and  }  [b]_{q,t}= \frac{q^b-t^b}{q-t}
\end{equation}

In \eqref{tm_sum}, the Hilbert series is over the space $DH_n$ which has dimension $(n+1)^{n-1}$. For more on this space, see \protect\cite {FBER,JHAG1}. Although the enumeration and asymptotics of Tesler matrices are not known, there are some nice product formulas when considering specializations of the alternating weight $\text{wt}_{q,t}(\cdot)$.  For instance, it was shown in \cite{AGHRS} that 
\begin{equation} \label{alter_result1}
q^{\binom{n}{2}} \sum_{A} \text{wt}_{q,q^{-1}} (A) = [n+1]_q^{n-1} 
\end{equation}
where $[n]_q=1+q+\cdots+q^{n-1}$. Furthermore, it was also shown in \protect\cite{PLEV} that 
\begin{equation} \label{alter_result2}
\sum_{A} \text{wt}_{q,0} (A) = [n]_q!
\end{equation}

Equations \eqref{alter_result1} and \eqref{alter_result2} show product formulas involving alternating sums of Tesler matrices. In this paper, we prove another such result that was initially conjectured by Armstrong in \protect\cite{DARM} by using a different alternating sum. He defines a \textit{poset} on the set of Tesler matrices which we will denote as $P(1^n)$ and refer to as the \textit{Tesler poset}. Recall that the \textbf{characteristic polynomial} on the poset $(P, \preceq)$, denoted $\chi(P;q)$, is a M\"{o}bius function weighted rank generating function. Hence, $$\chi(P;q)= \sum\limits_{A \in P} \mu(\hat{0},A) q^{\rho(P)-\rho(A)}$$ where we use the terminology and notation of \protect\cite[Ch.3]{RSTAN} for the  M\"{o}bius function $\mu(\cdot)$, the rank of an element $A \in P$ and of a poset $P$ as $\rho(A), \rho(P)$ respectively, and $\hat{0}$ for the unique least element. We will look at the characteristic polynomial of the Tesler poset, but we first need to give necessary definitions and conventions to discuss Tesler matrices in a precise manner. 

Let $U_n$ be the set of $n \times n$ upper-triangular matrices with non-negative integer entries. Given $A\in U_n$, where $A=(a_{i,j})$, we define the \textbf {hook sum} $h_k$ for $1 \leq k \leq n$ as follows:
$$h_k := ({a_{k,k}} + {a_{k,k+1}} + \cdots + {a_{k,n}} ) - ({a_{1,k}}  +  {a_{2,k}}  + \cdots  + {a_{k-1,k}}  ) $$
We define the \textbf {hook sum vector} as the $n$-dimensional vector $(h_1,\ldots,h_n)$. A \textbf{Tesler matrix} $A \in U_n$ is such that $h_k=1$ for all $1 \leq k \leq n$. 

\begin{ex}\label{ex:first}
The matrix below is a $3 \times 3$ Tesler matrix as $h_3=2-1-0=\bf{1}$, $h_2=1+1-1=\bf{1}$, and $h_1=0+0+1=\bf{1}$.

\[ \left( \begin{array}{ccc}
0 & 1 & 0 \\
&1 & 1\\
& &2\end{array} \right)\] 
\end{ex}

\medskip  
We denote the number of matrices in $U_n$ with a hook sum vector of $(\alpha_1, \ldots ,\alpha_n)$ as $T(\alpha_1, \ldots ,\alpha_n)$ and the set of such matrices as  $\mathcal{T}(\alpha_1, \ldots ,\alpha_n)$ and refer to these as \textbf{generalized Tesler matrices}. We often use short hand of $T(1^n)$ and $\mathcal{T}(1^n)$ for the number of and set of Tesler matrices respectively.

\begin{conj} [Armstrong \protect\cite{DARM}] \label{conj:armstrong}
Let $P(1^n)$ be the poset on Tesler matrices $\mathcal{T}(1^n)$, then $$\chi( P(1^n);q) = (q-1)^{ \binom{n}{2}} $$
\end{conj}

The method that we use in this paper extends to the larger class of generalized Tesler matrices with binary hook sums and settles Armstrong's conjecture with a simple calculation. 

\begin{thm} \label{generalized_armstrong_conj}
Let $\alpha =(\alpha_{n-1}, \ldots, \alpha_0) \in \{0,1\}^n$ and $P(\alpha)$ be the poset on generalized Tesler matrices $\mathcal{T}(\alpha)$. Then, letting $w(\alpha)= \sum_{i=0}^{n-1} i\cdot \alpha_i$, we have that $$  \chi(P(\alpha);q) = (q-1)^{w(\alpha)} $$  
\end{thm}

To see why this theorem settles Armstrong's conjecture, note that $w(1,1,\ldots,1) = \binom{n}{2}$. In addition, this theorem is also consistent with a well known result on the Boolean lattice (see Prop. \ref{boolean_lattice_bijection}). In order to prove this theorem, we will adapt a method \protect\cite{HSAG} of Joshua Hallam and Bruce Sagan. We also show that certain powers of $(q-1)$ divide the characteristic polynomial of the Tesler poset corresponding to a hook sum vector with either a trailing or a leading binary word. (See Corollary \ref{cor:binarywords}.)

\medskip
Although Tesler matrices have been connected in \protect\cite{JHAG1} to diagonal harmonics via a bivariate weight and in \protect\cite{MMR} were shown to be a solution to the {\textit Kostant partition function}, there are still many enumerative questions on Tesler matrices that have yet to be answered. For example, at the start of this paper, the previously known bounds for $T(1^n)$ were $n! \leq T(1^n) \leq 2^{\binom{n}{2}}$ \protect\cite[\S 4]{MMR}. In Section \ref{sec:armpoly}, through simple observations of an enumerative tool that we call the {\it Armstrong polynomial}, we are able to improve the lower bound such that $$ T(1^n) \geq (2n-3)!! $$ In addition, we can similarly get a tighter upper bound. There are also interesting enumerative results when considering generalized Tesler matrices. Let $C_i=\frac{1}{i+1}\binom{2n}{n}\sim \sfrac{4^n}{\left(\sqrt\pi n^{\frac{3}{2}}\right)}$ be the $i$th Catalan number. Zeilberger \protect\cite{DZE} showed that $$T(1,2,\ldots, n) =\prod\limits_{i=1}^n {C_i}$$ Thus $ T(1,2,\ldots, n)= e^{\Theta(n^2)}$, which motivated the following question.   

\begin{question}[Pak]\label{question:pak} True or False: The number of Tesler matrices have the following asymptotics
$$T(1^n)= e^{\Theta(n^2)}$$
\end{question}

\begin{remark}
Note that even the improved lower bound needs to be significantly improved further to give an affirmative answer to Question \ref{question:pak}. However, the existing data in the OEIS \href{http://oeis.org/A008608}{A008608} suggests that $log(T(1^n)) = O(n^{1.6})$ as noted in \protect\cite{AHMJP}.
\end{remark}

We denote the hook sum vector $(1,1,\ldots,1,0,0, \ldots, 0)$ with $k$ $1$'s and $(n-k)$ $0$'s as $(1^k,0^{n-k})$. This set of generalized Tesler matrices have previously been studied in \protect\cite{HIO} and we analyze the set $\mathcal{T}(1^k,0^{n-k})$ in Section \ref{sec:gen_tm} to get some insight into Tesler matrices. We will show that $$ T(1^k,0^{n-k}) \geq (k+1)^{n-1} \text{  for sufficiently large $n$} $$ 

This leads us to conjecture that the number of Tesler matrices can eventually be bounded below by the dimension of $DH_n$, which is $(n+1)^{n-1}$ (also the number of parking functions of size $n$). We also find generating functions $T_k(x)$ for particular values of $k$. When $k=1$, $T(1,0^{n-1}) = 2^{n-1}$, so this generating function is trivial. However, when $k=2$ we find the generating function in Proposition \ref{prop:genfun2} \protect\cite{HIO}. While the case where $k=3$ is still open, these generating functions could provide insight about a generating function for the number of Tesler matrices.
 
\medskip
\noindent \textbf{Outline:} 
 In Section \ref{sec:Background Information}, we will highlight some previous results and methods that will be pertinent in this paper. Then, in Section \ref{sec:Tesler_poset}, we introduce the Tesler poset, some its properties, and show that a specific hook sum vector yields a poset which is isomorphic to the well-known Boolean lattice that was initially noticed by Alejandro H. Morales in \protect\cite{AHM}. Using these results, we will then prove Theorem \ref{generalized_armstrong_conj} in Section \ref{sec:mainthm} and explore some of its corollaries. Finally, in the last two sections, we will explore asymptotics and other enumerative questions regarding generalized Tesler matrices and also explore the significance of settling Conjecture \ref{conj:armstrong} in respect to the asymptotics of Tesler matrices. 

\section{Background} \label{sec:Background Information}

\subsection{Tesler Generating Algorithm} \label{tesler_generating_process}
We will discuss a method for generating generalized Tesler matrices as given by Drew Armstrong \protect\cite{DARM}. Fix a generalized Tesler matrix $A=(a_{ij})$ of size $n$ with a hook sum vector $(\alpha_1,\ldots,\alpha_n)$. Then, consider the main-diagonal entries of $A$ as an $n$-tuple $(d_1,\ldots,d_n)$ with $d_i:=a_{ii}$. We will create a generalized Tesler matrix $A'=({a_{ij}}')$ with hook sum vector $(\alpha_1,\ldots,\alpha_n, \alpha_{n+1})$ by first constructing its main-diagonal $({d_1}',\ldots,{d_{n+1}}')$. For each entry $d_i$ in the $n$-tuple, we take that entry and replace it with ${d_i}'$ where $0\leq {d_i}' \leq d_i$ and set ${a_{n+1,i}}'=d_i-{d_i}'$ such that the $i$th hook sum remains unchanged. Then, let ${d_{n+1}}'$ be such that the sum of our newly constructed main-diagonal $(n+1)$-tuple adds up to $\sum\limits_{k=1}^{n+1} \alpha_k$ and let the other entries in the matrix remain unchanged. 
 
\begin{ex}
Our initial Tesler matrix has a main-diagonal tuple $(0,1,2)$. The algorithm generates six $(1 \cdot 2 \cdot 3)$ main-diagonal $4$-tuples and hence six Tesler matrices of size $4$.

\begin{figure}[H]
\centering
\includegraphics[width=100mm, scale=.75] {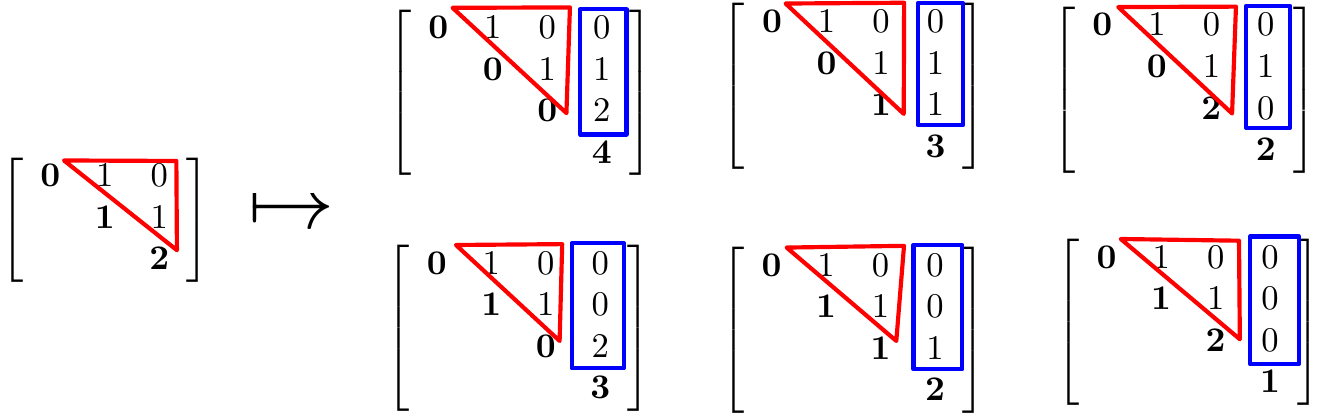}
\caption{Note that the red triangle is constant and that the blue rectangle corresponds to what was subtracted from the original main-diagonal.}
\label{Armstrong_generating)process.pdf}
\end{figure} 

\end{ex}

\begin{prop}
Iterating the Tesler Generating Algorithm yields all Tesler matrices. 
\end{prop} 

\begin{proof}
Seeking a contradiction, suppose that there exists a least integer $z$ corresponding to the size of at least one Tesler matrix $A$ that is not generated by this process. By reversing this process, we can then create a Tesler matrix of a smaller size $(z-1)$ that must be generated from this process as it is smaller in size than A. We could then generate $A$ from a matrix that is generated through this process. Hence, this process generates all of the Tesler matrices. 
\end{proof}

Fixing $A=(a_{ij})$ with hook sum vector $(\alpha_1,\ldots, \alpha_n)$, we now consider the number of generalized Tesler matrices of size $(n+1)$ that $A$ generates. 
\begin{defn}
Let $A=(a_{ij})$ be an $n \times n$ generalized Tesler matrix, then let $d_i:=a_{ii}$ be the $i$th main-diagonal entry. We define the \textbf{diagonal product} of $A$, or $\text{dpro}(A)$, as follows: 
$$\text{dpro}(A) = \prod\limits_{i=1}^n {(d_i +1)}$$
\end{defn}
 
\noindent Note that $$T(\alpha_1,\ldots, \alpha_n, \alpha_{n+1})=  \sum\limits_{A \in \mathcal{T}(\alpha_1,\ldots, \alpha_n)} {\text{dpro}(A)}$$

\subsection{Integral Flow Representation} \label{sec:ifr}
A Tesler matrix of size $n$ can also be represented as an integral flow on the complete directed graph on $(n+1)$ vertices with net flows equal to $(1,1,\ldots, 1, -n)$ \protect\cite{MMR}. Given any generalized Tesler matrix with hook sum vector $(\alpha_1, \ldots ,\alpha_n)$, we can represent it as an integral flow with net flows equal to $(\alpha_1, \ldots ,\alpha_n, -\sum\limits_{i=1}^n {\alpha_i} )$.

The bijection in \protect\cite{MMR} shows that these are equivalent notions. They consider the main-diagonal entry in row $i$ to be the flow sent from the {\it ith} vertex to the $(n+1)st$ vertex, which is the rightmost vertex. Then for each entry such that $i<j$, $a_{ij}$ corresponds to the flow between the {\it ith} and {\it jth} vertices. See Figure \ref{tesler_polytope_ex.pdf} below for an example of this bijection. 

\begin{figure}[H]
\centering
\includegraphics[width=50mm, scale=.25] {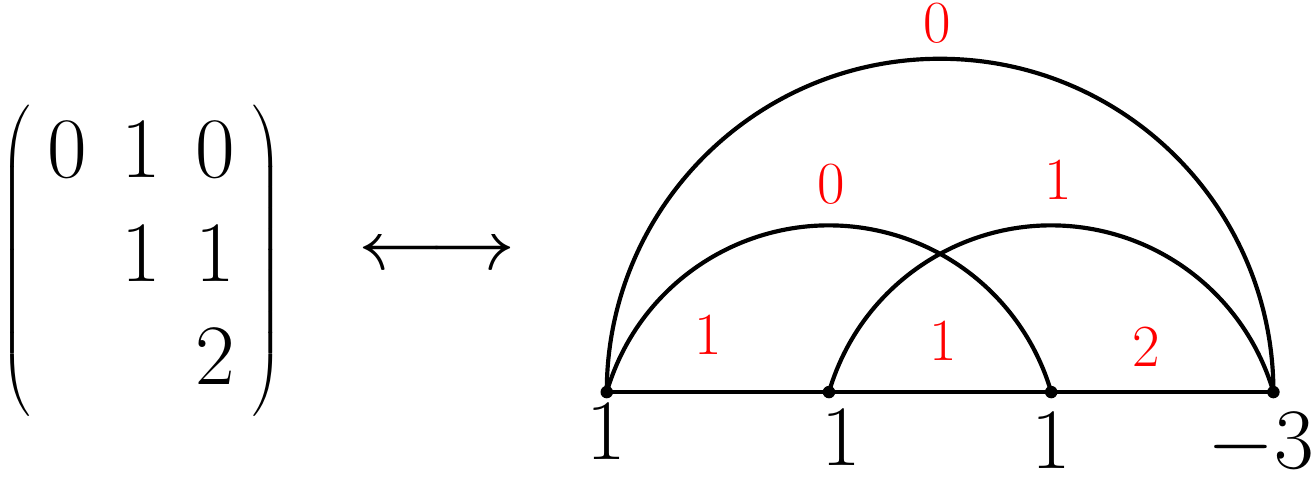}
\caption{Net flows depicted underneathe the complete directed graph}
\label{tesler_polytope_ex.pdf}
\end{figure}

\subsection{Method of Hallam and Sagan} \label{subsec:hallam_sagan} 

Sagan \protect\cite{BSAG} has previously done work on why the characteristic polynomial of a poset factors.  Recently, Sagan and Hallam \protect\cite{HSAG} have introduced a method for showing that the characteristic polynomial of a poset factors. We will apply Hallam and Sagan's method to the Boolean lattice to prove Theorem \ref{generalized_armstrong_conj}. Their method is to take ranked posets $P_1,\ldots, P_k$ for which the characteristic polynomial is known, and to consider $Q= P_1 \times \cdots \times P_k$. We recall the following facts regarding the characteristic polynomial of posets. 

\medskip
1) If $P \cong P'$, then $\chi(P;q) = \chi(P';q)$ 
 
2) $\chi(P_1 \times P_2; q) = \chi(P_1;q) \cdot \chi(P_2;q)$
 
\medskip
\noindent Then, they define an equivalence relation $(\sim)$ to identify elements in $Q$ such that  $\sfrac{Q}{\sim} \cong P$. The process of identifying elements leaves the characteristic polynomial unchanged if the equivalence relation satisfies certain conditions. First, they say that their equivalence relation is \textbf{homogeneous} if

\medskip
1) $\hat{0} \in Q$ is in an equivalence class by itself

2) If $X \geq Y$ in $\sfrac {Q}{\sim}$, then for all $x \in X$, there is a $y \in Y$ such that $x \geq y$. 
\medskip
\newline
Next, we need $\sim$ to preserve rank so that if $x \sim y$, then $\rho(x) = \rho(y).$ Lastly, letting $\mu(\cdot)$ be the M\"{o}bius function on $\sfrac {Q}{\sim}$ and considering any nonzero $X \in \sfrac {Q}{\sim}$ with lower order ideal $L(X) \subseteq \sfrac {Q}{\sim}$, 
\begin{equation} \label{summation_eq}
\sum\limits_{Y \in L(X)} \mu(\hat{0}, Y) = 0 
\end{equation}
Hallam and Sagan refer to \eqref{summation_eq} as the summation condition and we adopt this same terminology. 

\begin{lemma}[Hallam and Sagan \protect\cite{HSAG}] \label{HS_method}
Let $P, Q$ be posets as above and $\sim$ be an equivalence relation on $Q$ which is homogeneous, preserves rank and satisfies the summation condition. Then  $$\chi(\sfrac {Q}{\sim} ; q) = \chi(Q ;q)$$ 
\end{lemma} 

\begin{remark}
Hence, given suitable $P_i$, we see that $\chi(P ;q)$ factors. In Hallam and Sagan's paper \protect\cite{HSAG}, they use claws $CL_n$ to construct their products. We will use the Boolean lattice to construct our product. 
\end{remark}

\section{The Tesler Poset}\label{sec:Tesler_poset}
 
We first define the cover relation, introduced by Drew Armstrong \protect\cite{DARM}, and will then use this definition to prove a couple of useful facts which yield some intuition regarding the Tesler poset.

\subsection{Definition of Tesler Poset}
There are two cases in the example in Figure \ref{matrix_cover_relation.pdf} of the cover relation for the matrix representation depending on the location of the entries.   

\begin{defn}
Fix a hook sum vector $\alpha$. Then $A= (a_{ij}) \in \mathcal{T}(\alpha) $ covers $B=( b_{ij} ) \in \mathcal{T}(\alpha)$ and we write $ B \preceq A $ if there exists  $i<j<k$ such that  $a_{ij}=b_{ij}+1$, $a_{jk}=b_{jk}+1$, and $a_{ik}=b_{ik}-1$ or if there exists $i<j$ such that $a_{ij}=b_{ij}+1$, $a_{jj}=b_{jj}+1$, and $a_{ii}=b_{ii}-1$
\end{defn} 

\begin{figure}[H]
\includegraphics[width=50mm, scale=.4] {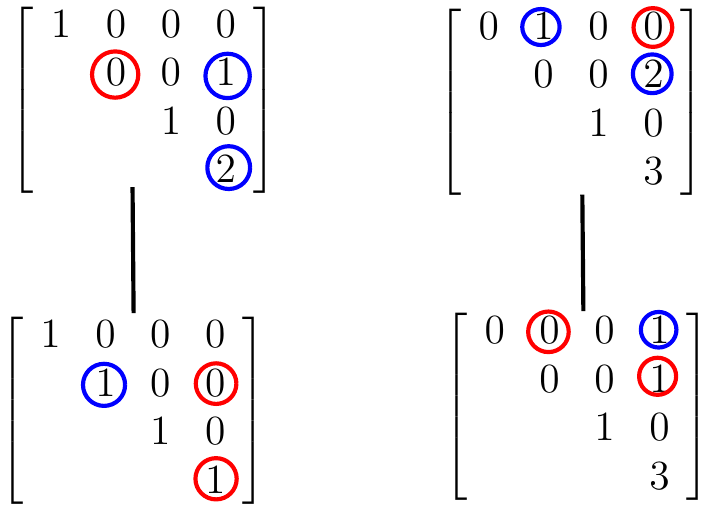}
\caption{The matrix version of the cover relation}
\label{matrix_cover_relation.pdf}
\end{figure}

The poset has a least element, $\hat{0}$, with the main-diagonal corresponding to the hook sum vector and all other entries equal to zero. Hence, in the case of a hook sum vector $(1,1,\ldots,1)$, the minimal element is the identity matrix of size $n$. 

\begin{remark}
With the equivalent notion of a Tesler matrix as an integral flow on the complete directed graph, the cover relation for the Tesler poset can also be described in terms of integral flows. Abusing notation, let $A,B$ be the corresponding integral flows to Tesler matrices $A$ and $B$ respectively. Then, integral flow $A$ covers $B$ if there exists vertices $i<j<k$ such that the flow between $i$ and $k$ is $1$ more in $B$ than it is in $A$ and the flow from vertices $j$ to $k$ and $i$ to $j$ is $1$ more in $A$ than it is in $B$.  
\end{remark}

\begin{figure}[H]
\centering
\includegraphics[width=25mm, scale=.5] {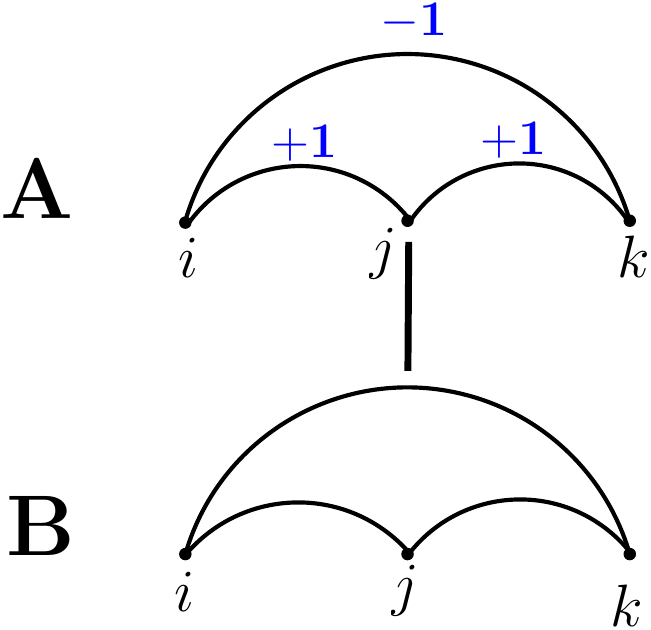}
\caption{The cover relation for the Integral flow representation}
\label{Tesler_cover_relation.pdf}
\end{figure}

\begin{ex} \label{ex:TeslerPoset3}
In the poset below, we see that Armstrong's conjecture is true for the case where $n=3$. Collecting terms from the bottom-up, we get $$\chi(P(1^3);q)=q^3-q^2-q^2-q^2+2q+q-1 = (q-1)^3$$ 
\end{ex}  

\begin{figure}[H]
\centering
\includegraphics[width=60mm, scale=.5] {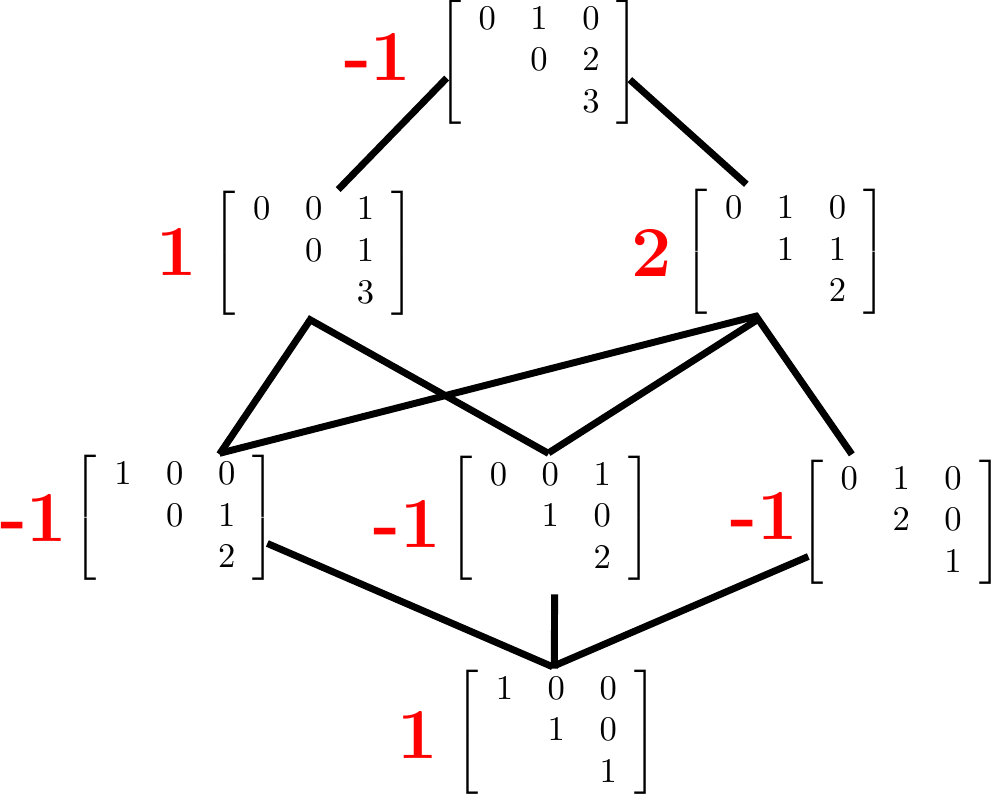}
\caption{The Tesler poset $P(1^3)$ with the values of the M\"{o}bius function in red. See appendix Figure \ref{tesler4poset.pdf} for the Hasse diagram of $P(1^4)$.}
\label{teslerposet.pdf}
\end{figure} 

\begin{remark}
By looking at the Hasse diagram of the Tesler poset $P(1^3)$ in Figure \ref{teslerposet.pdf}, we see that it is not a lattice. 
\end{remark}

\subsection{Properties}
We will show a few properties of the Tesler poset $P(\alpha)$ for $\alpha=(\alpha_1,\alpha_2, \ldots, \alpha_n)$.

\begin{prop}
The rank of a matrix in the Tesler poset $P(\alpha)$ is equal to the sum of the non-main-diagonal entries. That is, for $A=(a_{ij}) \in P(\alpha)$, we have that $$ \rho(A) = \sum\limits_{i>j} a_{ij}$$ 
\end{prop}

\begin{proof}
As we see in the definition of the cover relation, for any $A,B \in \mathcal{T}(\alpha)$, if $A$ covers  $B$, then we necessarily have that the sum of the non-main-diagonal entries for $A$ is one more than the sum of the non-main-diagonal entries for $B$. The minimal entry has a non-main-diagonal sum of $0$ and we get the desired result.    
\end{proof}

\begin{corollary}
The rank of the Tesler poset $P(\alpha)$ is $ \sum_{i=1}^{n} (n-i)\alpha_i $
\end{corollary}

\begin{proof}
The maximal element, $M \in P(\alpha)$, is such that the entry $M_{i,i+1} = \sum_{k=1}^i \alpha_k$ and $M_{n,n}=\alpha_n$ with all other entries zero. This is easy to see when considering the integral flow representation. The result then follows from the previous proposition.  	 
\end{proof}

\subsection{Relation to Boolean Lattice}
We now relate the poset formed by generalized Tesler matrices with hook sum vector $a_n= (1,0,\ldots,0,1)$ to the well-known Boolean lattice for subsets of $[n]:=\{1,2,\ldots,n \}$ under the inclusion relation. 

\begin{remark}
By the algorithm in Section \ref{tesler_generating_process}, the last element in the hook sum vector does not impact the poset. Hence, $P(1,0,\ldots,0,1) \cong P(1,0,\ldots,0)$.
\end{remark}

\begin{prop} \label{boolean_lattice_bijection}
Let $a_n = (1,0,\ldots,0,1)$ be a hook sum vector of size $n$. Then we have that $P(a_n) \cong B_{n-1}$  \protect\cite{AHM}.
\end{prop}

\begin{proof}
Let $\text{Pw}([n-1])$ be the power set of $[n-1]$. We define a bijection $\Phi: \mathcal{T}(a_n) \mapsto \text{Pw}([n-1])$ such that if there is a non-zero entry in the $(n-i+1)st$ column of $A\in \mathcal{T}(a_n)$, then $i \in \Phi(A)$ and otherwise $i \notin \Phi(A)$. (See Figure \ref{Boolean_Lattice_map.pdf} for an example of the map $\Phi$.) We show that this is a bijection via induction and the Tesler generating algorithm described in Section \ref{tesler_generating_process}. This clearly holds for when $n=1$, so let us suppose it holds for size $n$. Then, the size $n+1$ generalized Tesler matrices with hook sum vector $a_{n+1}$ are constructed from the size $n$ generalized Tesler matrices via the Tesler generating algorithm. 

\medskip
Consider an arbitrary generalized Tesler matrix with hook sum vector $a_n$ and by our inductive hypothesis, it corresponds with some subset of $[n-1]$. If we choose to subtract the one non-zero main-diagonal entry, and hence add it to the $(n+1)st$ column, then this amounts to adding the element $n$ to our set. If we don't subtract the non-zero main-diagonal element, then the $(n+1)st$ column will be all zeroes and hence, we don't add the element $n$ to our set. Therefore, this map is a bijection.

\begin{figure}[H]
\centering
\includegraphics[width=75mm, scale=.5] {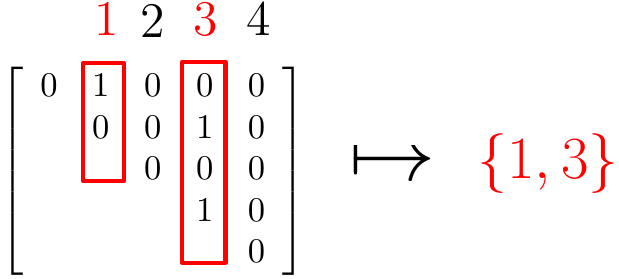}
\caption{An example of the bijection above for the case where $n=5$ and $k=3$}
\label{Boolean_Lattice_map.pdf}
\end{figure} 

Observe that this bijection is order preserving as all of the covers of a generalized Tesler matrix with hook sum vector $a_n$ correspond to a set which contain the initial set. This is because if there is a non-zero entry in the $ith$ column where $i \neq 1$, then there must be another non-zero entry in that same column after applying the cover relation since the hook sum must add up to $0$, applying the cover relation does not remove any of the elements of the set. If $i=1$, then there is only one such matrix that can satisfy the hook sum vector $a_n$ and this matrix corresponds to the empty set. As a result, $\Phi$ is also order preserving and we have our desired result. 
\end{proof}

\begin{corollary}\label{cor:boolean_lattice}
The characteristic polynomial of the poset $P(a_n)$ is $(q-1)^{n-1}$

\end{corollary}

\begin{proof}
The characteristic polynomial of $B_n$ is known to be $(q-1)^{n}$, hence the previous proposition that $P(a_n)\cong B_{n-1}$ gives us the desired result. 
\end{proof}

\section{Application of Hallam-Sagan to the Tesler Poset} \label{sec:mainthm}

\subsection{Initial Case} 
We now can use the Hallam-Sagan method discussed in Section \ref{subsec:hallam_sagan} for calculating the characteristic polynomial of the Tesler poset. In this subsection, we consider the initial case which serves as a motivating example. Let $\alpha,e_{n-1} \in \{0,1\}^n $ be such that $\alpha_{n-1}=0$, where $e_i$ is the $i$th elementary vector. In Figure \ref{product_method.pdf}, for instance, we have $\alpha=(1,0,1)$ and $\alpha + e_2 = (1,1,1)$. We want to compute the characteristic polynomial for the poset $P(\alpha + e_{n-1})$ using the characteristic polynomials of $P(\alpha)$ and $B_1$. 

We construct our product poset by considering a set of maps between $\mathcal{T}(\alpha)$ and  $\mathcal{T}(\alpha + e_{n-1})$. Let $\phi_1, \phi_2: \mathcal{T}(\alpha) \rightarrow \mathcal{T}(\alpha + e_{n-1})$ be such that $\phi_1 :A \mapsto A + \varepsilon_{n-1,n-1}$ and $\phi_2 :A \mapsto A + \varepsilon_{n,n-1} +\varepsilon_{n,n}$ where $\varepsilon_{i,j}$ is the elementary matrix. It is easy to check that these maps are well-defined and that they form a poset isomorphic to $B_1$ where $\phi_1 \preceq \phi_2$. In this motivation example, we define our equivalence relation $\sim$ on the product poset $P(\alpha) \times B_1$ as $(A,\phi_1) \sim (B, \phi_2)$ if and only if $\phi_1(A) = \phi_2(B)$. As we will show in Section \ref{gencasethm}, $\sim$ satisfies all of the  conditions in Lemma \ref{HS_method} so $$ \chi(\sfrac{ P(\alpha) \times B_1} {\sim};q) = \chi( P(\alpha) \times B_1;q) = \chi(P(\alpha);q) \cdot \chi(B_1;q) = (q-1)^{n-1}\cdot(q-1) = (q-1)^n$$

\begin{figure}[H]
\centering
\includegraphics[width=125mm, scale=1] {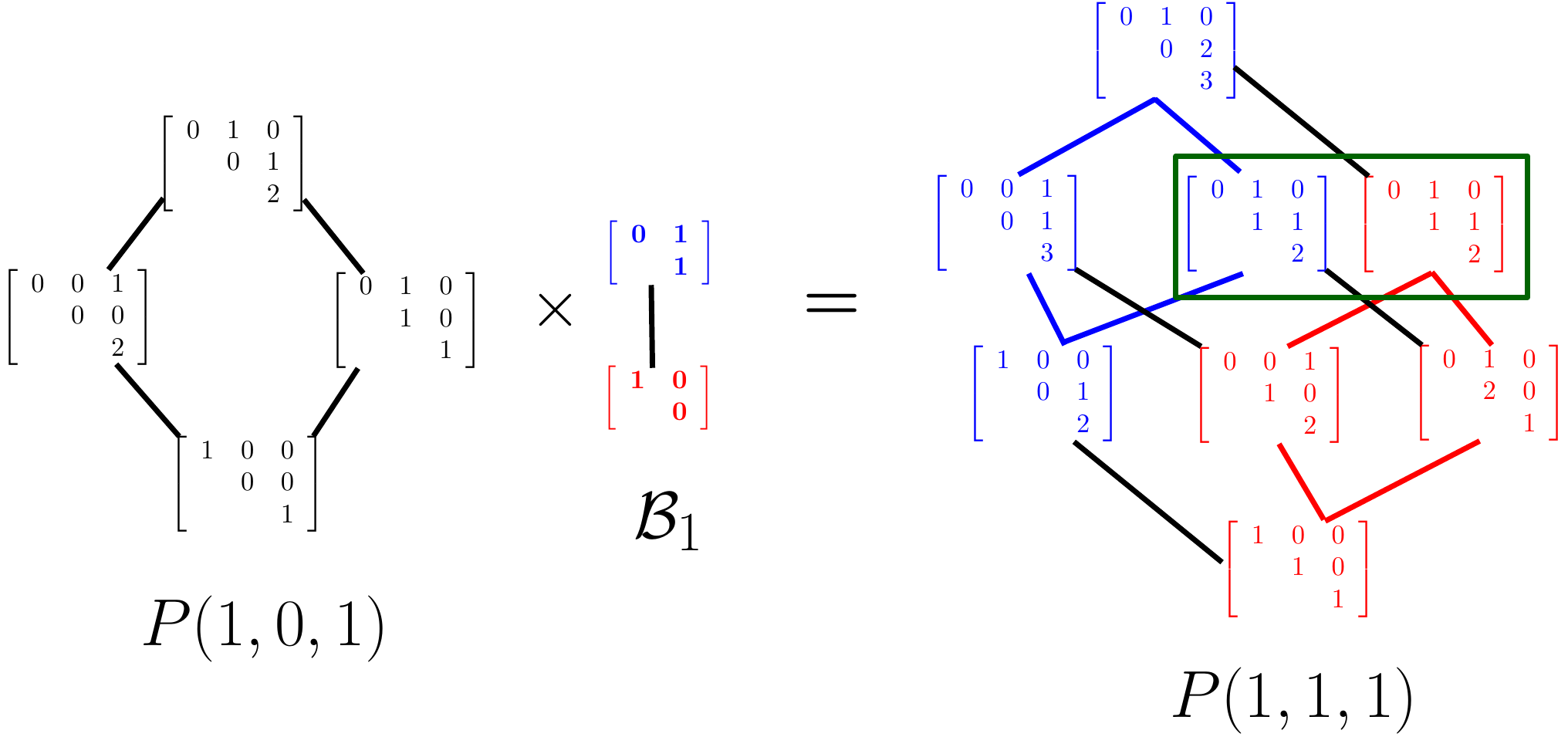}
\caption{Our method in the case where $n=3$ with the equivalent elements enclosed in a green rectangle.}
\label{product_method.pdf}
\end{figure}

\subsection{General Case} \label{gencasethm}
We will now generalize the idea from the previous section which will lead to our main theorem. Fix $n,r \in \mathbb{N}$ such that $r<n$, and let $\alpha \in \{0,1\}^n$ and $\alpha + e_{n-r+1} \in \{0,1\}^n$. The previous section considers the case where $r=2$. We seek to show that $$ \chi( P(\alpha + e_{n-r+1});q) = (q-1)^{r-1} \chi( P(\alpha);q) $$  We will consider a poset of maps from $\mathcal{T}(\alpha)$ to $\mathcal{T}(\alpha+e_{n-r+1})$. While there are certainly other such maps, we will consider a natural, intuitive set of maps which have a nice structure and turn out to be sufficient. In order for  $\phi:\mathcal{T}(\alpha) \mapsto \mathcal{T}(\alpha + e_{n-r+1})$ to be well-defined, it must increase the $(n-r+1)$st hook sum by $1$ while not changing the other hook sums. As a result, we consider maps which  can be thought of as an $r \times r$ upper triangular matrix with a hook sum vector $(1,0^{r-1})$, which is exactly an element of $\mathcal{T}(1,0^{r-1})$. We previously showed that these matrices are isomorphic to the Boolean lattice, so we often label these maps with their corresponding set.

\begin{ex}
Below is the poset of maps in the case where $r=3$. This poset is isomorphic to subsets of $\{1,2\}$ under the inclusion relation. 
\begin{figure}[H]
\centering
\includegraphics[width=75mm, scale=.5] {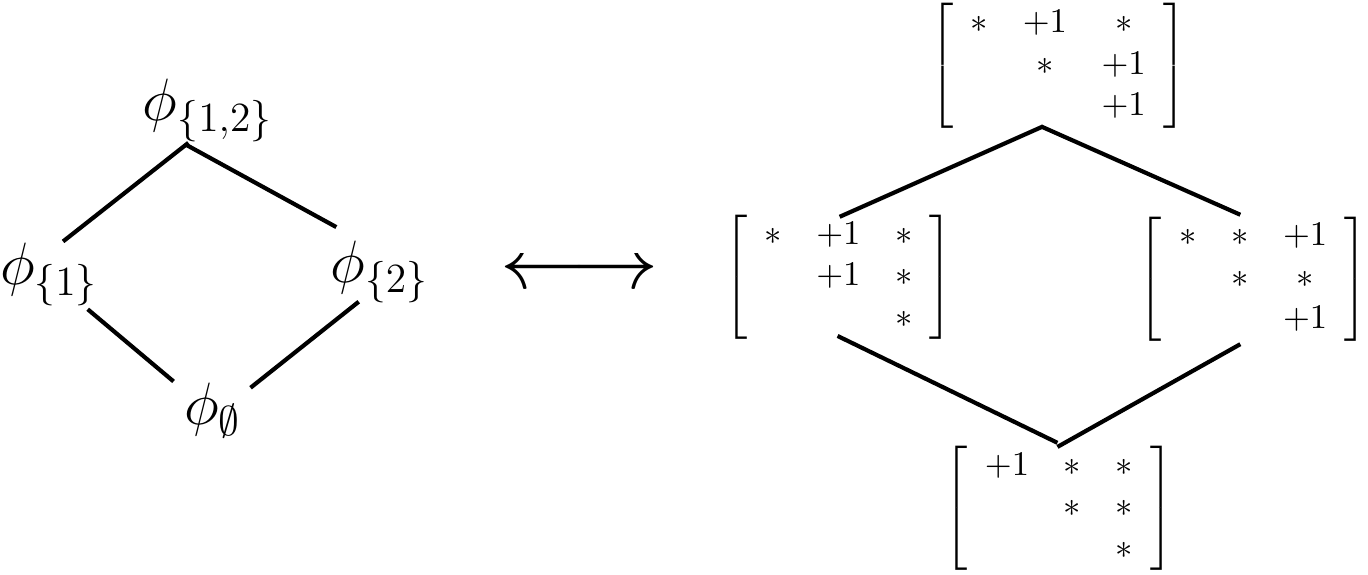}
\label{Boolean_Lattice_Maps.pdf}
\caption{* indicates no change to the element}
\end{figure} 

\end{ex}

\noindent Let $Q_A$ be the subposet of $P(\alpha+e_{n-r+1})$ of the matrices $\phi_A(\mathcal{T}(\alpha))$. 
\begin{prop}
We have the following facts: \\
(1) Let $A \subseteq [n]$, then $Q_A \cong P(\alpha)$ \\ 
(2) $\bigcup\limits_{A \subseteq [n]} \phi_A(\mathcal{T}(\alpha))=\mathcal{T}(\alpha + e_{n-r+1})$ 
\end{prop}

\begin{proof}
(1) Clearly $\phi_A$ is an injective map and is order preserving, so the posets are therefore isomorphic.  

\medskip
\noindent(2) By the well-defined nature of all of these maps, we clearly have that \newline $\bigcup\limits_{A \subseteq n} \phi_A(\mathcal{T}(\alpha)) \subseteq \mathcal{T}(\alpha + e_{n-r+1}) $. Now, let us consider the other direction. Let $A \in \mathcal{T}(\alpha + e_{n-r+1})$, then there must be a non-zero element in the $rth$ row. If this nonzero element is also in the $rth$ column, then one can check that $A \in \phi_{\emptyset}(\mathcal{T}(\alpha))$. Otherwise, by considering the columns with non-zero entries, we can construct a set $B \subseteq [n]$ in the same manner as we did in Proposition \ref{boolean_lattice_bijection} such that $A \in \phi_{B}(\mathcal{T}(\alpha))$. That is, if and only if there is a non-zero entry in the $(n-i+1)$st column, then the element $i$ is in the set $B$. As a result, we get that $ \mathcal{T}(\alpha + e_{n-r+1}) \subseteq \bigcup\limits_{A \subseteq n} \phi_A(\mathcal{T}(\alpha))$.  
\end{proof}

We can form a poset of the maps $\phi_A$ that is isomorphic to the Boolean lattice as we did in our motivating example and can consider the product poset $P(\alpha) \times B_{r-1}$. Since the maps $\phi_A$ are additive maps, we often view $\phi_A$ as a matrix $S_A \in \mathcal{T}(1,0^{r-1})$. 

\begin{defn}\label{defn:equiv}
We define the equivalence relation $\sim$ on $P(\alpha) \times B_{r-1}$ by $$(A_1,S_1) \sim (A_2, S_2) \text{  if  } A_1+S_1 = A_2+S_2$$ where the equality is matrix equality. We have to be careful with how we define the addition of these matrices as the dimensions of the square matrices do not match. We extend the matrix $S_i$ such that it is an $n \times n$ matrix in the following manner. The entry $S_{i,j}$ becomes the entry $S_{i+(n-r), j+(n-r)}$ and all other entries of $S$ are zero. Essentially, we are placing our matrix in the lower right corner in order to make addition of matrices defined. 
\end{defn}

\begin{figure}[H]
\centering
\includegraphics[width=90mm, scale=1] {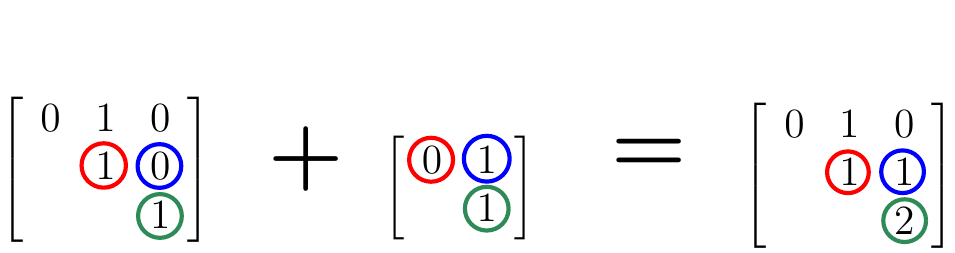}
\caption{The entries circled with the same color are added together to get our resulting $(A+S)$ matrix.}
\label{shifted_matrix_addition.pdf}
\end{figure}

Clearly this is a homogeneous equivalence relation which preserves rank as it satisfies the conditions discussed in Section \ref{subsec:hallam_sagan}. Therefore, we have that $\sfrac {P(\alpha) \times B_{r-1}}{\sim}$ is a valid poset. We now seek to show that the summation condition \eqref{summation_eq} holds. In order to do this, we will first need some technical lemmas. The first lemma restricts what elements can be in the same equivalence class. 

\begin{lemma} \label{first_sum_lemma}
Let $A_0$ be the minimal element of $P(\alpha)$, and $S_0$ be the matrix representation of $\phi_\emptyset$ which is the minimal element of $B_{r-1}$. Then, for non-minimal $A \in P(\alpha)$ and non-minimal $S_d \in B_{r-1}$, we have that  $(A_0,S_d)$ and $(A,S_0)$ are in different equivalent classes of the relation $\sim$.  
\end{lemma}

\begin{proof} 
We show that $A_0+S_d \neq A+S_0$ by showing they are not equal in the $(n-r+1)$st entry along the main-diagonal. That is, the values $(A_0 +S_d)_{(n-r+1,n-r+1)}$ and $(A+S_0)_{(n-r+1,n-r+1)}$ are different. On the RHS, we must have that this entry is equal to $0$ as it is $0$ in both matrices that we are adding. We know that this entry in $S_d$ is $0$ since otherwise we would necessarily have that $\phi_d$ is the minimal element. Considering this same entry for the LHS, we know that since $\phi_0$ has a $1$ in this particular entry, the non-negativity of elements in $A$ gives that this element on the LHS must be greater than or equal to $1$. Hence, we do not have matrix equality with the sum and thus the two elements are not equivalent under $\sim$.  See Figure \ref{Equivclass_lemma.pdf} below for a visual representative of this argument in a particular case.  
\end{proof}

\begin{figure}[H]
\centering
\includegraphics[width=100mm, scale=1] {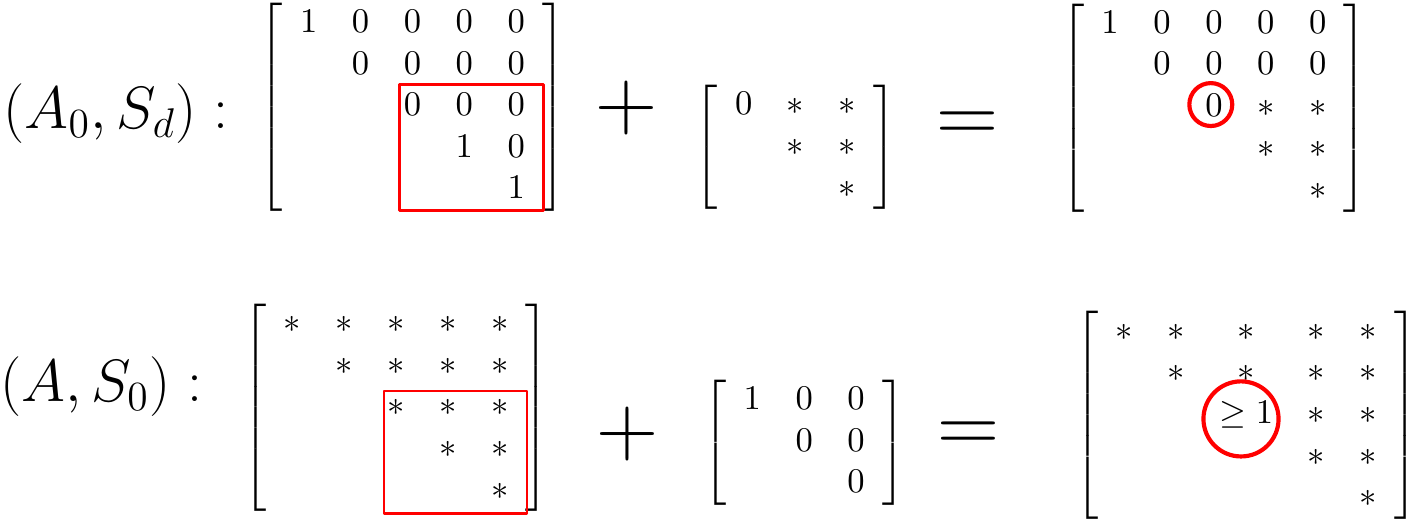}
\caption{The case when $n=5$ and $r=3$. Note that * indicates that we do not know the particular entry.}
\label{Equivclass_lemma.pdf}
\end{figure}

\noindent Now, we fix an element $X \in \sfrac {P(\alpha) \times B_{r-1}}{\sim}$. The next lemma dictates the elements that can be in the lower order ideal $L(X)$. For the rest of the section, we let $A_0$ be the minimal element of $P(\alpha)$ and $\phi_0$ be the minimal element of $B_{r-1}$.

\begin{defn}
 We say an element $A \in P(\alpha)$ is \textit{first coordinate isolated} if $(A, \phi_0$) is the only element in $L(X)$ with the first coordinate equal to $A$. Similarly, we say $\phi_d  \in B_{r-1}$ is \textit{second coordinate isolated} if  $(A_0, \phi_d )$ is the only element in $L(X)$ with the second coordinate equal to $\phi_d$.
\end{defn}

\begin{lemma} \label{second_sum_lemma}
At most one of the conditions hold. \\ 
(1) There exists a non-minimal $A \in P(\alpha)$ that is first coordinate isolated. \\ 
(2) There exists a non-minimal $\phi_d \in B_{r-1}$ that is second coordinate isolated.
\end{lemma}

\begin{proof}
We proceed by contradiction. Suppose that $\phi_d \in B_{r-1}$ is second coordinate isolated and $A\in P(\alpha)$ is first coordinate isolated. Now, since $(A, \phi_0)$ and $(A_0, \phi_d)$ are in $L(X)$, there exists a path in $L(X)$ between those elements and a member of the equivalence class $[X]$. By a path, we mean a sequence of covers in the poset. Consider such a path $\Gamma_1: (A, \phi_0) \mapsto (M_l, \phi_l) \sim X$. Note that this path stays in $L(X)$ and is guaranteed to exist by the fact that we are looking at elements in a lower order ideal. (See Figure \ref{path_contradiction_poset.pdf} for a pictorial representation of these paths.)

We start at $(A, \phi_0)$. In the first cover in our path, we necessarily must change the first coordinate. This is because if we were to change the second coordinate, we would get that $(A, \phi_1) \in L(X)$ which contradicts our hypothesis that $A\in P(\alpha)$ is first coordinate isolated. Therefore, our first cover in $\Gamma_1$ must be $(A, \phi_d) \mapsto (A_1, \phi_0)$ for some $A_1 > A \in P(\alpha)$. Now, suppose our second cover resulted in our path $\Gamma_1$ going to $(A_1, \phi_1)$ for some $\phi_1 \in B_{r-1}$. This would imply that $(A_1, \phi_1) \in L(X)$ and hence $(A, \phi_1) \in L(X)$ as this is a lower order ideal. This contradicts our hypothesis that $A\in P(\alpha)$ is first coordinate isolated. Hence, our updated path is  $ \Gamma_1: (A, \phi_d) \mapsto (A_1, \phi_0) \mapsto (A_2, \phi_0)$ for some $A_2 > A_1 \in P(\alpha)$. Continuing this argument, we see that our path $\Gamma_1$ must have a constant second coordinate $\phi_0$. As a result, we have that $X \sim (M_l, \phi_l) \sim (A_m, \phi_0) $ where $A_m > \cdots > A_1 > A \in P(\alpha)$ for some $m \in \mathbb{N}$.  

Now, by a similar argument, we have that $\Gamma_2: (A_0, \phi_d) \mapsto (M_j, \phi_j)$ must have constant first coordinate $A_0$. As a result, we have that $X \sim (M_j, \phi_j) \sim (A_0, \phi_{d_k}) $ where for some $k \in \mathbb{N}$ $\phi_{d_k} > \cdots > \phi_{d_1} > \phi_d \in B_{r-1}$ . Note that transitivity implies that $(A_m, \phi_0) \sim (A_0, \phi_{d_k})$.  In Lemma \ref{first_sum_lemma}, we showed that these are necessarily in different equivalence classes, hence we have reached our contradiction.
\end{proof}

\begin{figure}[H]
\centering
\includegraphics[width=60mm, scale=1] {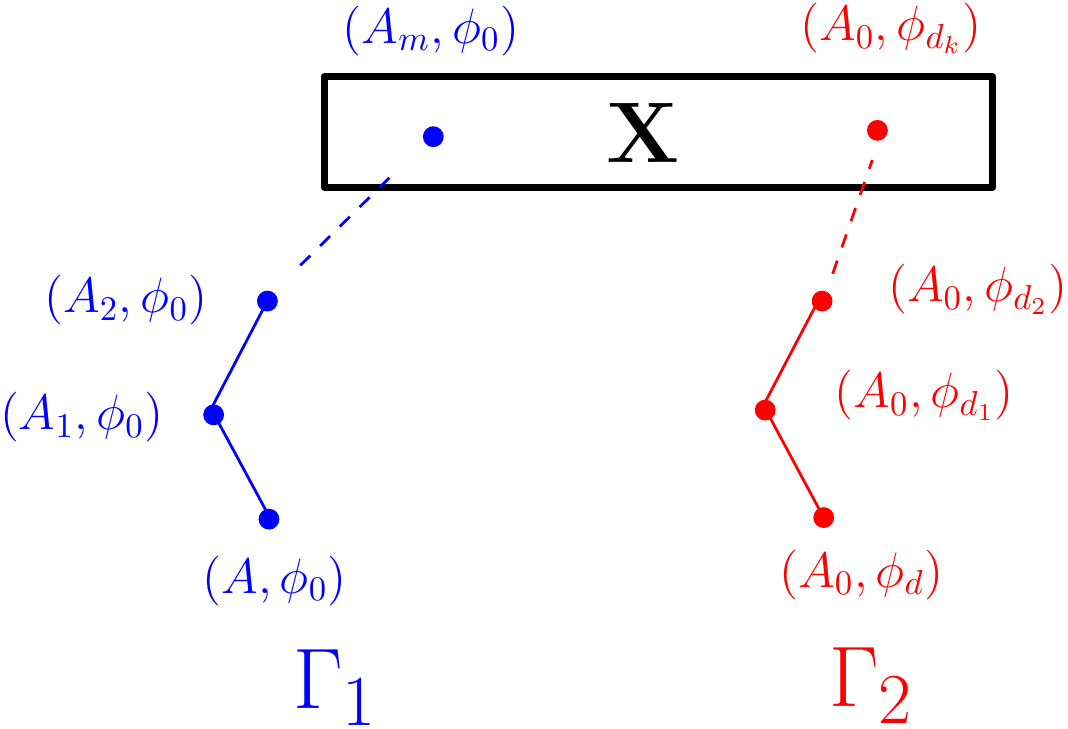}
\caption{Pictorial Representation of paths $\Gamma_1, \Gamma_2$ in proof of Lemma \ref{second_sum_lemma} }
\label{path_contradiction_poset.pdf}
\end{figure}

\begin{prop}
The summation condition \eqref{summation_eq} holds for the poset $\sfrac {P(\alpha) \times B_{r-1}}{\sim}$
\end{prop}

\begin{proof}
Fix a non-minimal equivalence class $X\sim[A,S] \in \sfrac {P(\alpha) \times B_{r-1}}{\sim} $. We must show that $$ \sum\limits_{(Y,S)\in L(X)} \mu((\hat{0},\hat{0}), (Y,S)) = 0 $$ Observe that we can write the LHS in the following two ways 

\begin{align}
& \sum\limits_{(Y,S)\in L(X)} \mu \bigl((\hat{0},\hat{0}), (Y,S)\bigr) = \sum_{S_i} \left( \sum\limits_{(Y,S_i)\in L(X)} \mu \bigl((\hat{0},\hat{0}\bigr), (Y,S_i)\bigr) \right) \label{eq:first_coordinate}  \\
& \sum\limits_{(Y,S)\in L(X)} \mu \bigl((\hat{0},\hat{0}), (Y,S)\bigr) = \sum_{Y_k} \left( \sum\limits_{(Y_k,S)\in L(X)} \mu\bigl((\hat{0},\hat{0}), (Y_k,S)\bigr) \right)\label{eq:second_coordinate}  
\end{align}
Now, since we are considering the lower order ideal of a product, it is easy to evaluate \newline $\sum\limits_{(Y_k,S)\in L(X)} \mu((\hat{0},\hat{0}), (Y_k,S))$. By the product structure of the lower order ideal $L(X)$, there is a unique maximum $Y_k \in P(\alpha)$ for elements in $L(X)$ with the second coordinate $S_i$. This follows by supposing that there are at least two incomparable relative maximal elements and using a very similar argument from the previous lemmas. By the recursive nature of the M\"{o}bius function, we get that so long as $Y_k$ is not the minimal element in $P(\alpha)$, the inner sum in Equation \eqref{eq:second_coordinate} is always $0$ and $$\sum\limits_{(Y_k,S)\in L(X)} \mu\bigl((\hat{0},\hat{0}), (Y_k,S)\bigr) =0$$ Similarly, so long as $S_i$ is not the minimal element in $B_{r-1}$ the inner sum in \eqref{eq:first_coordinate} is always $0$ so that $$\sum\limits_{(Y,S_i)\in L(X)} \mu\bigl((\hat{0},\hat{0}), (Y,S_i)\bigr) = 0$$  

Thus, it suffices to show that we do not have a $A\in P(\alpha)$ which is first coordinate isolated and a $\phi_d \in B_{r-1}$ which is second coordinate isolated. We showed this precise statement in Lemma \ref{second_sum_lemma}. Hence, \eqref{summation_eq} holds as we are either adding up all zeroes in \eqref{eq:first_coordinate} or \eqref{eq:second_coordinate}.  
\end{proof}

\noindent We are now ready to prove the lemma that we use in our main theorem. 

\begin{lemma} \label{inductive_step_method}
Let $\alpha =(\alpha_1, \ldots, \alpha_n) \in \{0,1\}^n$ where $\alpha_{n-r+1}=0$. Then  $$ \chi( P(\alpha + e_{n-r+1});q) =  \chi( P(\alpha);q) \cdot (q-1)^{r-1} $$
\end{lemma}

\begin{proof}
It suffices to show that $P(\alpha + e_{n-r+1}) \cong \sfrac {P(\alpha) \times B_{r-1}}{\sim}$. We have already shown that there is a bijection between the elements of the poset. We now must show that this bijection is order-preserving. This follows by the fact that both sets have the same Tesler cover relation and by the definition of cover in a product poset. In the forward direction, this follows immediately by the definition of cover in a product poset. In the other direction, if we have a cover in $P(\alpha + e_{n-r+1})$, then we must have a non-zero element in the same spot in one of the two coordinates, and hence can create a cover in this coordinate.   
Thus, we have that  $P(\alpha + e_{n-r+1})\cong \sfrac {P(\alpha) \times B_{r-1}}{\sim}$, and using Lemma \ref{HS_method} we get
\begin{align*}
\chi(P(\alpha + e_{n-r+1});q) &= \chi( \sfrac {P(\alpha) \times B_{r-1}}{\sim};q) \\ 
&= \chi(P(\alpha)\times B_{r-1};q) \\ 
&= \chi(P(\alpha);q) \cdot \chi(B_{r-1};q) = \chi(P(\alpha);q) \cdot (q-1)^{r-1}  
\end{align*}
\end{proof}
\noindent We are now ready to state and prove our main theorem. Note that we have a slight modification in our notation for the hook sum vector $\alpha$ for a more clean result. 

\begin{thm} \label{generalized_theorem}
Let $\alpha =(\alpha_{n-1},\alpha_{n-2}, \ldots, \alpha_1, \alpha_0) \in \{0,1\}^n$ where $\alpha_{n-1}=\alpha_0=1$. Then, letting $w(\alpha)= \sum_{i=0}^{n-1} i\cdot \alpha_i$ we have that $$ \chi(P(\alpha);q) = (q-1)^{w(\alpha)}$$  
\end{thm}

\begin{proof}
We iterate Lemma \ref{inductive_step_method} for each $\alpha_i=1$ where $i \in [2,n-1]$. Note that if $\alpha_i =0$, we are not changing the poset, so the characteristic polynomial is unchanged. One way of representing this using Lemma \ref{inductive_step_method} is to multiply by $(q-1)^{\alpha_i(n-i)}$. This multiplies the characteristic polynomial of the unchanged poset by $1$ when $\alpha_i=0$ and by the desired amount when $\alpha_i=1$. We start with the hook sum vector $\alpha_{n-1} + \alpha_0$ and then apply the Lemma \ref{inductive_step_method} to get the characteristic polynomial for  $\alpha_{n-1} + \alpha_{1} +  \alpha_0$ as we did in our motivating example. We then do the same thing to get $\alpha_{n-1} + \alpha_{2} +\alpha_{1} +  \alpha_0$, and iterate until we have the characteristic polynomial of the poset corresponding to the hook sum vector $\alpha$.  As a result, we get that 
\begin{align*}
\chi( P(\alpha_{n-1} + \alpha_{1} +  \alpha_0);q) &= (q-1)^{n-1} \cdot (q-1)^{\alpha_1} \\
\chi( P(\alpha_{n-1} + \alpha_2 + \alpha_{1} +  \alpha_0);q) &= (q-1)^{n-1} \cdot (q-1)^{\alpha_1} \cdot (q-1)^{2\alpha_2} \cdot  \\
&\text{  }  \vdots \\
\chi(P(\alpha);q) &= (q-1)^{n-1} \cdot (q-1)^{\alpha_1} \cdots (q-1)^{(n-2)\alpha_{n-2}} \\
\end{align*}
Collecting powers we obtain $(q-1)^{w(\alpha)}$ as desired.

\end{proof}

\begin{corollary}\label{cor:settled_armconj}
Let $P(1^n)$ be the Tesler poset and $w(\alpha)$ be as above. Then $$\chi( P(1^n);q) = (q-1)^{\binom{n}{2}} $$
\end{corollary}

\begin{proof}
Since $w(1^n)= \binom{n}{2} $, the result follows by Theorem \ref{generalized_theorem}.
\end{proof}

Note that Theorem \ref{generalized_theorem} also generalizes the well-known result on the Boolean lattice result as the Boolean lattice is isomorphic to the Tesler poset $P(1,0,\ldots, 0)$. We see this by noting that $$ w(1,0,\ldots,0) = w(1,0,\ldots,0,1) = n-1 $$

\begin{remark}
This result also gives another method of generating Tesler matrices $\mathcal{T}(1^n)$ that is different from the Tesler generating algorithm discussed in Section \ref{tesler_generating_process}. While this method is certainly less efficient that the Tesler generating algorithm, it is possible to construct the set $\mathcal{T}(1^n)$ in this manner without knowledge of the sets $\mathcal{T}(1^1), \ldots, \mathcal{T}(1^{n-1})$ and only using the well known Boolean lattice. 
\end{remark}

A natural question is to see if this result extends to all generalized Tesler matrices. In the general case, Lemma \ref{first_sum_lemma} and Lemma \ref{second_sum_lemma} do not hold. For other $\alpha \in \mathbb{N}^n$, we get other factors besides $(q-1)$ as we see in \eqref{eq:otherfactors}. Moreover, in the general case, the characteristic polynomial need not factor over $\mathbb{Z}$ as we see in \eqref{eq:doesnotfactor}.

\begin{ex}
Let $\alpha_1=(1,2,3)$ and $\alpha_2=(2,1,1,1)$ and consider the posets $P(1,2,3)$ and $P(2,1,1,1)$. (For the Hasse diagram of the poset $P(1,2,3)$, see Figure \ref{tesler121poset.pdf} in the Appendix.) Then, one can check that   
\begin{equation}\label{eq:otherfactors}
\chi(P(\alpha_1);q) = q(q-1)^3 
\end{equation}

\begin{equation}\label{eq:doesnotfactor}
\chi(P(\alpha_2);q)= (q-1)^4 (q^5-2q^4+4q^3-6q^2+3q+1)
\end{equation}
\end{ex}

\noindent  However, we do have the following divisibility results as corollaries to Theorem \ref{generalized_theorem}. The question of $(q-1)$ divisibility in the Tesler poset was initially considered by Drew Armstrong and then communicated in \protect\cite{DARM}.

\begin{corollary}\label{cor:leadingbinword}
Let $\alpha \in \mathbb{N}^{k}$ and consider the Tesler poset $P(1,\alpha)$, then $$ (q-1)^{k} \text{ divides } \chi(P(1,\alpha) ;q) $$ 
\end{corollary}

\begin{proof}
We start off with the posets $P(\alpha)$ and $B_{n-1}$ and consider the product $P(\alpha) \times B_{n-1}$ and apply the same equivalence relation from Definition \ref{defn:equiv} and note that the results from  Lemma \ref{first_sum_lemma} and Lemma \ref{second_sum_lemma} also hold in this case. As a result, we can use Lemma \ref{HS_method} and Proposition \ref{boolean_lattice_bijection} to get that  $$\chi(P(1,\alpha) ;q)= \chi( \sfrac{ B_{n-1} \times P(\alpha)}{\sim};q) = \chi(B_{n-1} \times P(\alpha) ;q) = (q-1)^{n-1} \cdot \chi(P(0,\alpha);q)$$  \end{proof}

We can now use Corollary \ref{cor:leadingbinword} and Lemma \ref{inductive_step_method} to get some results about factors of the characteristic polynomial when there are leading and trailing binary words in the hook sum vector.

\begin{corollary}\label{cor:binarywords}
Let $\alpha \in \mathbb{N}^{n-k}$ and $\beta=(\beta_1, \ldots, \beta_k) \in \{0,1 \}^k$ and consider the Tesler posets $P(\alpha, \beta)$ and $P(\beta, \alpha)$. Then, letting $w_1(\beta) = \sum\limits_{i=1}^{k} (n-i)\beta_i $ and $w_2(\beta)=  \sum\limits_{i=1}^{k} (k-i)\beta_i $: 

\begin{equation}\label{eq:leadingzeroes}
 (q-1)^{w_1(\beta)} \text{ divides } \chi(P(\beta,\alpha) ;q)
 \end{equation}

\begin{equation}\label{eq:trailingzeroes}
(q-1)^{w_2(\beta)} \text{ divides } \chi(P(\alpha,\beta) ;q)
\end{equation}
\end{corollary}

\begin{proof}
First, we look at the statement in \eqref{eq:leadingzeroes}. We iterate through our binary word $\beta$ by starting with $\beta_k$ and ending with $\beta_1$. At the $i$th step, by Corollary \ref{cor:leadingbinword}, we get a factor of $(q-1)^{n-k+i-1}$. Collecting powers of $(q-1)$, and then reordering the sum gives us the desired result of a factor of $  (q-1)^{w_1(\beta)}$.
Next, we consider the statement in \eqref{eq:trailingzeroes}. We iterate through our binary word $\beta$ by starting with $\beta_1$ and ending with $\beta_k$ and use the result from Lemma \ref{inductive_step_method}. After collecting powers, we get a factor of $(q-1)^{w_2(\beta)}$.
\end{proof}

\section{Armstrong polynomial}\label{sec:armpoly}
In this section, we introduce the Armstrong polynomial to encode the growth of the number of Tesler matrices. Questions regarding asymptotics of the Kostant partition function and hence generalized Tesler matrices have recently appeared in a {\tt Math Overflow} article. \protect\cite{DSTE} \newline  Let $A= (a_{i,j}) \in \mathcal{T}(1^n)$ and $d_i=a_{i,i}$. Recall the diagonal product of $A$ $$\text{dpro}(A) = \prod\limits_{i=1}^n {(d_i +1)}$$.  

\begin{defn}
We define the Armstrong polynomial $A_n(q)$ to measure the distribution of diagonal products in $\mathcal{T}(1^n)$. That is,  
$$A_n(q) := \sum\limits_{A\in T_n} q^{\text{dpro}(A)}$$ 
\end{defn}

\begin{ex}
As we see in Figure \ref{teslerposet.pdf}, the diagonals corresponding to the $7$ Tesler matrices of size $3$ are $(1,1,1), (0,1,2), (0,0,3),(1,0,2), (0,0,3), (0,1,2), \text{ and } (0,2,1)$ with diagonals products: $8, 6, 4, 6, 4, 6,\text{ and }6$. As a result, we have  $$A_3(q) = 2q^4 + 4q^6 + q^8$$
\end{ex}

\subsection{Previously known bounds}
It follows that for all $A \in \mathcal{T}(1^n)$, we have \newline $ n+1 \leq \text{dpro}(A) \leq 2^n$ where the tightness of the lower and upper bound are obtained with main-diagonals $(0,0, \ldots, 0, n)$ and $(1,1,\ldots,1)$ respectively. The first approximation considers {\it all} the diagonal products to be $(n+1)$ to get the lower bound and $2^n$ to get the upper bound. Through this method we get that 
\begin{equation}\label{eq:first_approx}
n! \leq T(1^n) \leq 2^{\binom{n}{2}}
\end{equation}

\begin{ex}
For n=1 through n=5, we have the following Armstrong polynomials:
\begin{align*} 
&A_1(q) = 1q^2 \\ 
&A_2(q) = 1q^3 + 1q^4  \\ 
&A_3(q) = 2q^4 + 4q^6 + 1q^8  \\ 
&A_4(q) = 7q^5 + 15q^8 + 6q^9 + 11q^{12} + 1q^{16}  \\
&A_5(q) = 40q^6 + 93q^{10} + 67q^{12} + 75q^{16} + 55q^{18} + 26q^{24} + 1q^{32} 
\end{align*}
\end{ex}

\begin{remark}
As we will discuss in Section \ref{sec:gen_tm}, we can also define the Armstrong polynomial $A_n(\alpha,q)$ for certain classes of generalized Tesler matrices with hook sum vector $\alpha$. 
\end{remark}

\begin{prop} \label{co:prop}
Let $[q^a]A_n(q)$ be the coefficient of the term of degree $a$ in $A_n(q)$, then 

\medskip
1) $[q^{2^n}] A_n(q) = 1$ 

\medskip
2) $[q^{n+1}] A_n(q) = T(1^{n-1})$.

\medskip
3) $[q^{3\cdot2^{n-2}}] A_n(q)= 2^n-n-1$
\end{prop}

\medskip
\begin{proof}
The first statement follows as there is only one diagonal, namely ${(1,1, \ldots , 1)}$, which results in a diagonal product of $2^n$ and the identity matrix is the only such matrix with this diagonal. 
For the second statement, the only possible main-diagonal with diagonal product $(n+1)$ is $(0,\ldots,0,n)$. Using the Tesler generating algorithm discussed in Section \ref{tesler_generating_process}, the only way to get such a diagonal is to start out with any main-diagonal of size $(n-1)$ and then taking everything away from all elements of the original diagonal. As a result, for each Tesler matrix of size $(n-1)$, we have a unique Tesler matrix of size $n$ with diagonal  $(0,\ldots,0,n)$, thus proving the second statement.

Finally, considering the last part of our proposition, let the coefficient of the term with degree $3\cdot2^{n-2}$ in $A_n(q)$ be $a_n$. We simply need to show that $a_n$ satisfies the same recurrence relation as the sequence $\{2^n-n-1\}$. Namely, we need to show that $a_n= (n-1) + 2a_{n-1}$. One can check that the terms with degree $3\cdot2^{n-2}$ in $A_n(q)$ come from the diagonal $(2,1,\ldots1,0)$ and valid rearrangements of those terms. Starting with diagonals in the form $(2,1,\ldots1,0)$ of the previous size, we can either do nothing, or subtract $2$ from the $2$ term in the diagonal $(2,1,\ldots1,0)$. This accounts for the $2a_{n-1}$. We get the $(n-1)$ from noting that we can also generate the diagonal $(2,1,\ldots1,0)$ by starting from the unique Tesler matrix with main-diagonal $(1,\ldots,1)$ and subtracting any one of the $(n-1)$ main-diagonal entries that are $1$.
\end{proof}

Note that given $k \in \mathbb{N}$ and the Armstrong polynomial, $A_k(q)$, it is possible read off $T(1^{k-1})$, $T(1^k)$, and $T(1^{k+1})$ from this polynomial as we show in the following proposition. 

\begin{prop}\label{prop:coeff_darmpoly}
The Armstrong polynomial has the following characteristics: 

\medskip
1) $T(1^{n+1}) = \frac{d}{dt}  { A_n(q)  |_{q=1} }  $

\medskip
2) $A_n(1) = T(1^n) $
\end{prop}

\begin{proof}
First, we note that \text{dpro}(A) $\geq$ 2 for all $n$. Then,
 
\begin{align*}
\frac{d} {dt}  { A_n(q)  \bigr|_{q=1} }  &=  \frac{d} {dt} { \sum\limits_{A\in \mathcal{T}(1^n) } q^{\text{dpro}(A)} \bigr|_{q=1} } = \sum\limits_{A\in \mathcal{T}(1^n)} {  \frac{d} {dt} {q^{\text{dpro}(A)} \bigr|_{q=1} } }  \\
 &= \sum\limits_{A\in \mathcal{T}(1^n)} { \text{dpro}(A) \cdot q^ { \text{dpro}(A) - 1 }  \bigr|_{q=1} } \\ 
 &= \sum\limits_{A\in \mathcal{T}(1^n)} { \text{dpro}(A) } = T(1^{n+1}).  
\end{align*}
The second statement is immediate.
\end{proof}

We can now use the observations in Proposition \ref{co:prop} regarding the Armstrong polynomial to get the following bounds on the number of Tesler matrices.
\begin{thm}
\begin{equation}\label{eq:newbounds}
\prod\limits_{i=1}^{n-1} (2i-1) \leq T(1^n) \leq 2^{\binom{n-2}{2}-1}\cdot 3^n
\end{equation}
\end{thm}

\begin{proof}
We use a similar method as we did in our first approximation. This time, however, we know that we have exactly $T(1^{n-1})$ of our terms to have a diagonal product of $(n+1)$ by Proposition \ref{prop:coeff_darmpoly}. We now assume that the remaining Tesler matrices have a diagonal product of $2n$, the second lowest diagonal product. Using this, we note that

\begin{align*} 
T(1^{n+1}) &= \sum\limits_{A \in \mathcal{T}(1^n)} \text{dpro}(A) \\
& \geq T(1^{n-1}) (n+1) + \left[T(1^n) - T(1^{n-1})\right] (2n) \\ 
\end{align*}

\noindent We now use the previously known bounds in \eqref{eq:first_approx} that $T(1^n) \geq n T(1^{n-1}) $ to get that 

\begin{align*}
T(1^{n+1}) &\geq T(1^n) \left[ \frac{T(1^{n-1})}{T(1^n)} (n+1) + \frac{T(1^n)-T(1^{n-1})}{T(1^n)} (2n) \right] \\ 
& \geq T(1^n) \left[ \frac{1}{n} (n+1) + \frac{n-1}{n} (2n) \right] \geq T(1^n) (2n-1) \\ 
\end{align*}

\noindent Iterating this, we get our desired lower bound that $$T(1^{n+1}) \geq \prod\limits_{i=1}^n(2i-1) = (2n-1)!! $$ 

\noindent We get the upper bound by the same method and further reductions. 
\end{proof}

\begin{remark}
We note that the lower bound in \eqref{eq:newbounds} is better than $n!$ since  $\prod\limits_{i=1}^{n-1} (2i-1) \geq 2^{n-2}\cdot(n-2)!$ and is $O( (\frac{2n}{e})^{n-1})$. Note that this still does not give an affirmative answer to Question \ref{question:pak} and that the upper bound in \eqref{eq:newbounds} is still $e^{\Theta(n^2)}$, but is slightly tighter.  
\end{remark}

\section{ Understanding Different Hook Sum Vectors } \label{sec:gen_tm}
Recall that $\mathcal{T}(1^k,0^{n-k})$ denotes the set of generalized Tesler matrices with hook sum vector equal to $(1,\ldots,1,0,\ldots,0)$ where there are $k$ $1$'s and $(n-k)$ $0$'s and that  $T(1^k,0^{n-k})$ denotes the number of such matrices. In this section, we will refer to Armstrong polynomials for generalized Tesler matrices with hook sum vector $\alpha$ as $A_n(\alpha,q)$ where the Armstrong polynomial from the previous section is such that $A_n(q):=A_n(1,1, \ldots,1,q)$.

First, we consider $\mathcal{T}(1,0^{n-1})$. It follows that  $\mathcal{T}(1,0^{n-1})$ can be generated by the method discussed in Section \ref{tesler_generating_process}. Now, we note that there is only one possible diagonal up to reordering of $(1,0, \ldots, 0)$, so all elements have the same diagonal product and as a result the Armstrong polynomial is always in the form $A_n(1,0, \ldots, 0,q)=(T(1,0^{n-2})) q^2 $ which yields that $T(1,0^{n-1}) = 2^{n-1}$ by the first part in Proposition \ref{prop:coeff_darmpoly}. Hence, letting $T_1(x)$ be the generating function for the number of generalized Tesler matrices with hook sum vector $(1,0^{n-1})$ we get that $$ T_1(x) = \frac{1}{1-2x} $$

Now, let us consider $\mathcal{T}(1^2,0^{n-2})$. These matrices have been recently studied in \protect\cite{CSKM, HIO}. For the same reason as above, we can consider the corresponding Armstrong polynomial. There are only two possible diagonals up to reordering of $(2,0, \ldots, 0)$ and $(1,1,0, \ldots, 0)$ with diagonal products $3$ and $4$ respectively. We now consider the corresponding Armstrong polynomial $A_n(1^2,0^{n-2},q)$
 
\begin{prop}
Let $A_{n-1}(1^2,0^{n-3},q) = a_{n-1}q^3 + b_{n-1}q^4$. Then, we have $$A_n(1^2,0^{n-2},q) = ( 2a_{n-1} + b_{n-1})q^3 + ( a_{n-1} + 3b_{n-1})q^4$$ 
\end{prop}

\begin{proof}
We only need to that prove the value of the coefficient of $q^3$ is as stated as the other coefficient is determined by the fact we know the total number of matrices that are in this set from the $A_{n-1}(1^2,0^{n-2},q)$ term. Thus, we consider the ways to get the diagonal $(2,0, \ldots, 0)$ from the previous set. First, we can do nothing in the diagonal part of the Tesler generating process and add a zero to each of the $(2,0, \ldots, 0)$ of the previous case. Second, for all of the previous size matrices, we subtract everything from the diagonal and then add $2$ yielding the diagonal $(0, \ldots, 0,2)$. As a result, we generate  $2a_{n-1} + b_{n-1}$ distinct terms with diagonal $(2,0, \ldots, 0)$.
\end{proof}

\begin{prop}
Let $t_n:=T(1^2,0^{n-2})$. Then $t_n \geq 3^{n-1}$ for $n\geq5$.
\end{prop}

\begin{proof} 
Generating these matrices through a computer program, we note that $t_5 = 90$. Thus, since $3$ is the smallest possible diagonal product we have  $t_n \geq 3t_{n-1} \cdots \geq 3^{n-5} t_5 = 90(3^{n-5} )\geq 3^{n-1} $
\end{proof}

\begin{prop}[See also \protect\protect\cite{HIO}]\label{prop:genfun2} 
 The ordinary generating function for $t_n$ is $$T_2(x) = \frac{ 1-4x-2x^2} {1-5x+5x^2}$$
\end{prop}

\begin{proof}
We note that $t_n$ has the following recurrence relation $t_{n+1} = 5t_{n} - 5t_{n-1}$. From this difference equation, and the initial conditions $t_1=1$ and $t_2=2$, we can find the generating function for $t_n$ via standard methods. 
\end{proof} 

\begin{prop}\label{prop:motivating_conj} 
For all k, there exists some $N_k \in \mathbb N$ such that for all $n \geq N_k$ we have $$T(1^k,0^{n-k} ) \geq (k+1)^{n-1} $$
\end{prop}

\begin{proof}
For a given $k$, the smallest possible diagonal product in $\mathcal{T}(1^k,0^{n-k} )$ is $(k+1)$. Using similar methods of generating the diagonals of the form $(0,0,\ldots,0,k)$, we can see that less than half of the terms in the set $\mathcal{T}(1^k,0^{n-k} )$ have a diagonal product of $(k+1)$. Hence, noting that the next lowest diagonal product is $2k$, the expected value of the diagonal product is {\it at least} $\sfrac{(3k+1)} {2}$. Since $\sfrac{(3k+1)} {(2k+2)} >1$  for  $k \geq 2$, we will eventually have an $N_k$ such that $T(1^k,0^{N_k-k} ) \geq (k+1)^{N_k-1}$
\end{proof}

\subsection{Conjectures and Future Work}

\noindent The sequence $\{T(1^n)\}$ appears in the OEIS \href{http://oeis.org/A008608}{A008608}. Based on the $25$ entries in this sequence, and the insight from Proposition \ref{prop:motivating_conj}, we make the following conjecture. 
\begin{conj}
Let $n,k \in \mathbb{Z}$ be such that $n \geq k \geq 11$. Then, we have $$T(1^k,0^{n-k}) \geq (k+1)^{n-1}$$ 
\end{conj} 

\begin{remark}
This conjecture would prove that for $n \geq 11$, we have $T(1^n) \geq (n+1)^{n-1}$ which is a significant because $(n+1)^{n-1}$ is the value of \eqref{tm_sum} with $t=1$ and $q=1$ (i.e. the dimension of $DH_n$). We note that for $k=11$, this conjecture is true as $T(1^{11})= 515,564,231,770$ which is bigger than $12^{10}$. Thus if we can show that $T(1^{n+1}) \geq e\cdot (n+2)T(1^n)$ for $k \geq 11$, then we have proven the conjecture. Here the number $e$ comes from looking at the fraction of the next term over the previous term which gives us $(\frac{n+2}{n+1})^{n-1}\cdot(n+2)$ where the first term is bounded below by $e$.
\end{remark}

The statistics dinv and area, which are mentioned in more detail in \protect\cite{JHAG1}, are used in the now settled Haglund-Loehr conjecture \protect\cite{JHAG3}. Carlsson and Mellit show in \protect\cite{ECAR} that 

\begin{equation}\label{eq:carlssonmellit}
\text{Hilb}(DH_n;q,t)= \sum\limits_\pi q^{\text{dinv}(\pi)}t^{\text{area}(\pi)}
\end{equation}

\noindent where the sum is over {\em parking functions} $\pi$ of size $n$. 

\noindent Haglund's Tesler matrix approach to showing \eqref{eq:carlssonmellit} reduces to proving that 
\begin{equation}\label{eq:hlconj}
\sum\limits_{\pi} q^{dinv(\pi)}t^{area(\pi)}=\sum\limits_{A=(a_{i,j})\in\mathcal{T}(1^n)} wt_{q,t}(A)
\end{equation}

\noindent where $\text{wt}_{q,t}(\cdot)$ is as in \eqref{eq:wt_def}. 

\medskip
\noindent It was shown in \protect\cite{AGHRS} that by plugging in $t=1$ and $q=1$ we get 
\begin{equation}\label{eq:pf_permtm}
(n+1)^{n-1} = \sum\limits_{A=(a_{i,j})\in \mathcal{PT}(1^n)} \prod\limits_{a_{i,j} > 0} [a_{i,j}]_{q,t}
\end{equation}

\noindent We note here that the only terms that survive on the RHS of \eqref{eq:hlconj} after plugging in $t=1$ and $q=1$ are Tesler matrices with exactly one nonzero element in each row. These are called {\it Permutation Tesler matrices}. This relationship between parking functions and Tesler matrices adds intrigue to having the number of parking functions eventually be a lower bound for Tesler matrices since this would imply there is a lot of cancellation in the alternating sum on the RHS of \eqref{eq:hlconj}. We will now explore a way to affirmatively answer Question \ref{question:pak} using $\chi(P(1^n);q)$. 
 
\begin{prop}
Let $\mu(\cdot)$ be the M\"{o}bius function for the Tesler poset $P(1^n)$. If for all $A \in \mathcal{T}(1^n)$ we have that $|\mu(\hat{0},A)| \leq f(n)$, then we have that:
\[ T(1^n) \geq \frac{2^{\binom{n}{2}}}{f(n)} \]
\end{prop}
 
\begin{proof}
We note that by Corollary \ref{cor:settled_armconj}, we have $$\sum\limits_{A} |\mu(\hat{0},A)|  \geq 2^{\binom{n}{2}}$$ Hence, if for all $A \in \mathcal{T}(1^n)$ we have that $|\mu(\hat{0},A)| \leq f(n)$, then we would have that\newline $T(1^n) \cdot f(n) \geq 2^{\binom{n}{2}}$ which gives the desired result. 
\end{proof}
 
\begin{remark}
We would find such a bound on the  M\"{o}bius function for the Tesler poset $P(1^n)$ by analyzing the size of the equivalence classes that we get when we use Hallam-Sagan's method from Section \ref{subsec:hallam_sagan}. In their Lemma \ref{HS_method}, they show that the M\"{o}bius function of the equivalence class $[X]$ is equal to the sum of the M\"{o}bius function evaluated at the elements in the  equivalence class $[X]$.

\begin{conj}\label{conj:mobfunctionbound}
Let $\alpha=(1,1,\ldots,1)$ and $P(\alpha)$ be the Tesler poset with M\"{o}bius function $\mu(\cdot)$. Then we can have the following lower bound on the  M\"{o}bius function $$ |\mu(\hat{0},A)| \leq n! $$ 
\end{conj}

We have been able to computationally able to verify Conjecture \ref{conj:mobfunctionbound} in the Tesler poset corresponding to hook sum vectors $(1,1,\ldots,1)$ up size $5$. A proof of this conjecture would give an affirmative answer to Question \ref{question:pak} since $$T(1^n) \geq \frac{2^{\binom{n}{2}}}{n!} = e^{\Theta(n^2)} $$
\end{remark}

\section{Appendix} 
\begin{figure}[H]
\centering
\includegraphics[width=150mm, scale=1] {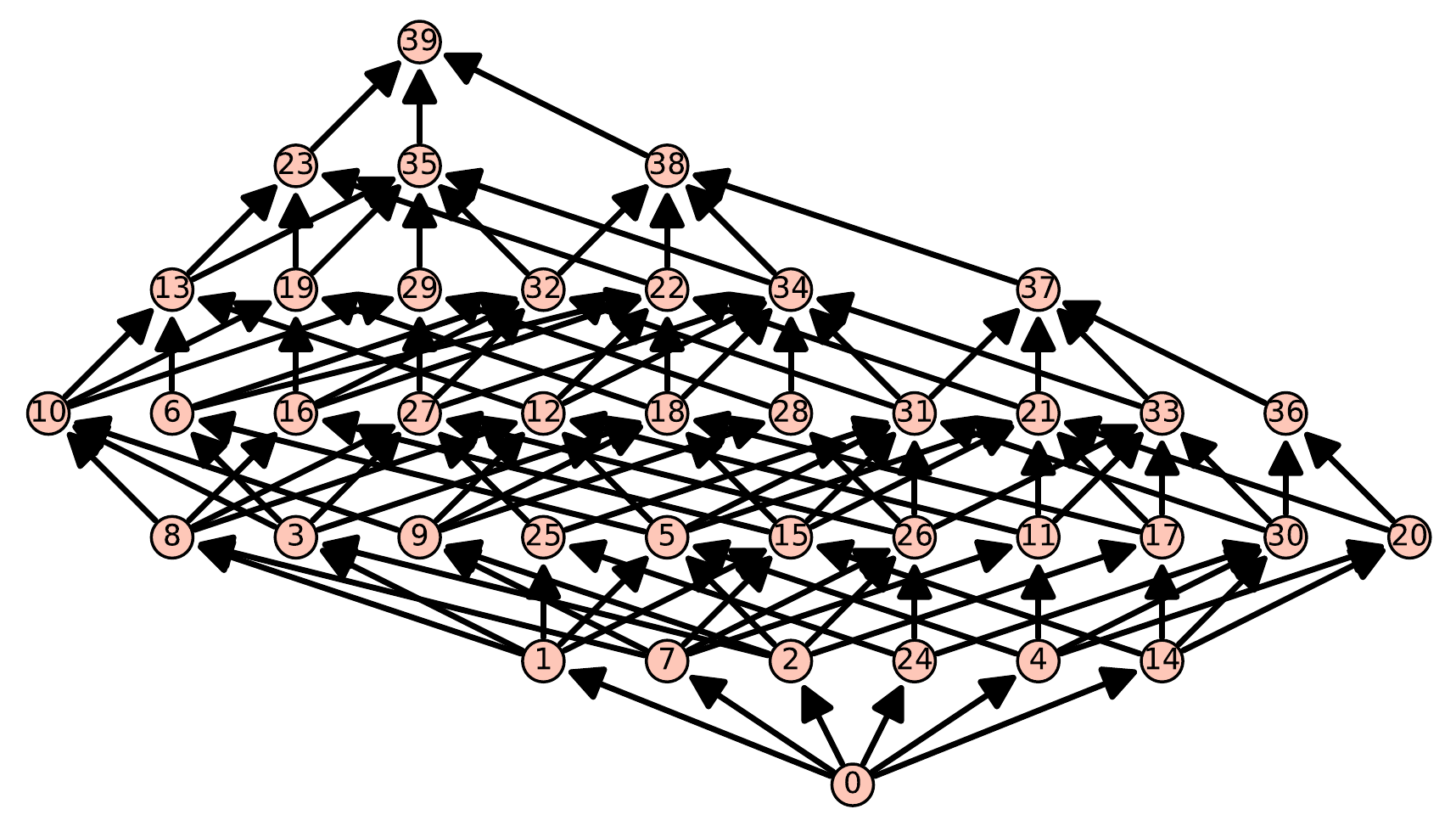}
\caption{The Tesler poset corresponding to hook sum vector $(1,1,1,1)$}
\label{tesler4poset.pdf}
\end{figure}

\begin{figure}[H]
\centering
\includegraphics[width=30mm, scale=.75] {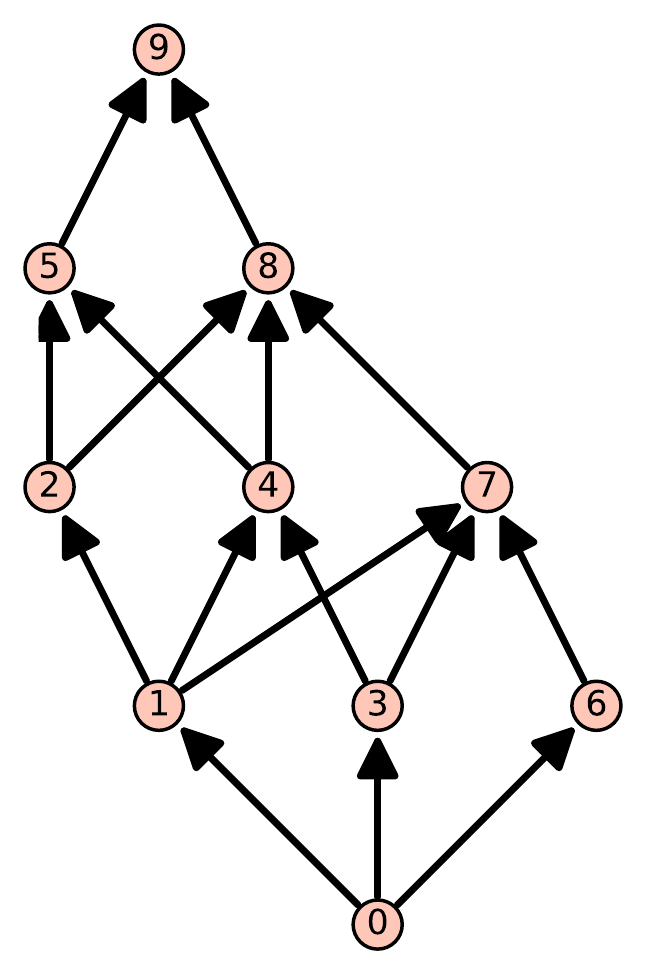}
\caption{The Tesler poset corresponding to hook sum vector $(1,2,3)$}
\label{tesler121poset.pdf}
\end{figure}

\section{Acknowledgements} 
I would like to thank Alejandro H. Morales for mentoring me during my REU, proposing the project, numerous edits in the process of writing this paper, and for suggesting the method of Hallam and Sagan in Section \ref{subsec:hallam_sagan}. I would also like to thank Igor Pak for supporting the REU and his question regarding the number of Tesler matrices and Drew Armstrong for conjecture \ref{conj:armstrong} and for the method of generating Tesler matrices in Section \ref{tesler_generating_process} in conversations with Alejandro H. Morales. Lastly, I want to thank the UCLA Math Department and the private donors to the VIGRE Pure Math REU for providing me with an Undergraduate Research experience.

\end{document}